\documentclass[reqno, 12pt]{amsart}

\usepackage{amsfonts, amsthm, amsmath, amssymb}
\usepackage{hyperref}
\hypersetup{colorlinks=false}

\usepackage[margin=1.5in]{geometry}

\usepackage{helvet}

\RequirePackage{mathrsfs} \let\mathcal\mathscr

\usepackage{hyperref}
\usepackage{dsfont}
\usepackage{upgreek}
\usepackage{mathabx,yfonts}
\usepackage{enumerate}

\numberwithin{equation}{section}

\newtheorem{theorem}{Theorem}[section] 
\newtheorem{lemma}[theorem]{Lemma}

\theoremstyle{definition}
\newtheorem*{acknowledgements}{Acknowledgements}

\newcommand{\eul}{\mathrm{e}}

\renewcommand{\phi}{\varphi}

\newcommand{\ZZ}{\mathbb{Z}}
\newcommand{\ZZp}{\mathbb{Z}_{\mathrm{prim}}}
\newcommand{\NN}{\mathbb{N}}
\newcommand{\QQ}{\mathbb{Q}}
\newcommand{\RR}{\mathbb{R}}
\newcommand{\CC}{\mathbb{C}}

\newcommand{\cM}{\mathcal{M}}

\newcommand{\cR}{\mathcal{R}}

\newcommand{\cP}{\mathcal{P}}

\renewcommand{\leq}{\leqslant}
\renewcommand{\le}{\leqslant}
\renewcommand{\geq}{\geqslant}

\renewcommand{\bar}{\overline}

\newcommand{\x}{\mathbf{x}}
\newcommand{\y}{\mathbf{y}}
\renewcommand{\c}{\mathbf{c}}

\renewcommand{\b}{\mathbf{b}}

\renewcommand{\r}{\mathbf{r}}

\newcommand{\fo}{\mathfrak{o}}
\newcommand{\fa}{\mathfrak{a}}
\newcommand{\fb}{\mathfrak{b}}
\newcommand{\fc}{\mathfrak{c}}
\newcommand{\fd}{\mathfrak{d}}
\newcommand{\fe}{\mathfrak{e}}

\newcommand{\fp}{\mathfrak{p}}
\newcommand{\fq}{\mathfrak{q}}

\newcommand{\ve}{\varepsilon}

\DeclareMathOperator{\vol}{vol}

\DeclareMathOperator{\res}{Res}

\DeclareMathOperator{\n}{N}

\DeclareMathOperator{\moo}{mod} 
\renewcommand{\bmod}[1]{\,(\moo{#1})}

\newcommand{\Zp}{\mathbb{Z}_{\text{prim}}}

\DeclareSymbolFont{bbold}{U}{bbold}{m}{n}
\DeclareSymbolFontAlphabet{\mathbbold}{bbold}

\newcommand{\Q}{\mathbb{Q}}

\newcommand{\N}{\mathbb{N}}
\newcommand{\R}{\mathbb{R}}

\newcommand{\Z}{\mathbb{Z}}
\renewcommand{\l}{\left}

\renewcommand{\r}{\right}
\renewcommand{\b}{\mathbf}
\renewcommand{\c}{\mathcal}
\renewcommand{\epsilon}{\varepsilon}

\renewcommand{\leq}{\leqslant}
\renewcommand{\geq}{\geqslant}
\renewcommand{\#}{\sharp}

\title[Arithmetic functions over principal ideals]{Averages of arithmetic functions\\ over principal ideals}

\author{T. D. Browning}
\address{School of Mathematics\\
University of Bristol\\ Bristol\\ BS8 1TW\\ UK}
\email{t.d.browning@bristol.ac.uk}

\author{E. Sofos}
\address{Max Planck Institute for Mathematics \\ 
Bonn \\ 
Vivatgasse 7\\
53111 \\ 
Germany} 
\email{sofos@mpim-bonn.mpg.de}

\date{\today}

\begin{document}

\begin{abstract}
For a general class of non-negative  functions
defined on integral ideals of number fields,
upper bounds are established for their average over the values of certain principal ideals that are  associated to irreducible binary forms with integer coefficients. 
\end{abstract}

\maketitle

\setcounter{tocdepth}{1}
\tableofcontents

\section{Introduction}\label{s:intro}

The study of averages of non-negative multiplicative arithmetic functions $f$ over the values of polynomials has a long and venerable history in number theory. 
From the point of view of upper bounds, this topic goes back to  work of Nair \cite{Nair}, which  has since been substantially generalised by 
 Nair--Tenenbaum \cite{NT}  and Henriot \cite{KH}.
For suitable expanding regions $\mathcal{B}\subset \ZZ^2$,
several authors have  extended these results to cover averages of the shape
$$
\sum_{(s,t)\in \mathcal{B}} f(|F(s,t)|),
$$
where $F\in \ZZ[s,t]$ is an  irreducible  binary  form.
This is the object of work 
by la Bret\`eche--Browning \cite{nair} and 
la Bret\`eche--Tenenbaum \cite{moyennes},
 for example.  These estimates have since had many applications to a range of problems, most notably in the quantitative arithmetic of Ch\^atelet surfaces \cite{annals}. 

Assuming for the moment that $F(x,1)$ is monic and irreducible, any root $\theta$ of the polynomial generates a number field $K=\QQ(\theta)$ whose degree is equal to the degree of $F$. 
In this paper we instead consider a variant in which 
we  take a general non-negative multiplicative function $f$ defined on the ideals of $K$, and ask to bound the size of the sum
$$
\sum_{(s,t) \in \mathcal{B}} f(s-\theta t),
$$
where 
we view $(s-\theta t)$ as an ideal in the ring of integers $\fo_K$ of $K$.
Our primary motivation for considering this sum is the crucial role that it plays in 
work of the authors \cite{dreamteam} on Manin's conjecture for smooth quartic del Pezzo surfaces.

In order to present our main result we require some notation and definitions.  
Let $K/\QQ$ be a number field with
ring of integers $\fo_K$.
Denote by $\c{I}_K$ the set of ideals in $\fo_K$. 
We say that a function $f:\c{I}_K\to \RR_{\geq 0}$ 
is {\em pseudomultiplicative} if 
there exist strictly positive constants $A,B,\ve$ such that   
$$
f(\fa \fb)
\leq 
f(\fa)
\min\left\{A^{\Omega_K(\fb)}, B(\n_K\fb)^\ve\right\},
$$
for all coprime ideals $\fa,\fb\in \c{I}_K$, where
$
\Omega_{K}(\fb)=\sum_{ \fp \mid  \fb}
\nu_\fp(\fb)
$
and $\n_K \fb=\#\fo_K/\fb$ is the ideal norm.
We denote the class of all
pseudomultiplicative functions 
associated to $A,B$ and $\ve$   
by  $\cM_K=\cM_K(A,B,\ve)$. 
When $K=\QQ$, this class contains the class of  
submultiplicative functions that arose in pioneering work of Shiu \cite{shiu}.
Note that any $f\in \c{M}_K$ satisfies 
$f(\fa) \ll A^{\Omega_K(\fa)}$ and  
$f(\fa)  \ll  (\n_K \fa)^\varepsilon$,  
for any  $\fa\in \c{I}_K$,
where the second implied constant depends on $B$.

We will need to work with functions supported away from  ideals of small norm.  To facilitate this, for any ideal $\fa\in \c{I}_K$ and $W\in \NN$, we set
$$
\fa_W=\prod_{\substack{\fp^\nu \| \fa \\ \gcd(\n_K \fp,W)=1}}
\fp^\nu.
$$
We extend this to rational integers in the obvious way.
Next,  for any $f\in \c{M}_K$, we define 
$
f_W(\fa)=f(\fa_W).
$
We will always assume that $W$ is of the form
\begin{equation}\label{eq:WW}
W=\prod_{p\leq w} p,
\end{equation}
for some $w>0$. Thus $\gcd(\n_K\fp,W)=1$ if and only if $p>w$, if $\n_K\fp=p^{f_\fp}$ for some $f_\fp\in \NN$.
Let
\begin{equation}\label{eq:span}
\!\cP_K^\circ\!=\!
\left\{\fa\subset \fo_K: 
\fp\mid \fa \Rightarrow f_{\fp}=1
\right\}
\end{equation}
be the
multiplicative span of all prime ideals $\fp\subset \fo_K$ with residue degree $f_\fp=1$. 
For any  $x>0$
and  $f\in \c{M}_K$ we 
 set  
$$
E_{f}(x;W)=\exp\Bigg(
\sum_{\substack{
\fp\in \cP_K^\circ \text{ prime}\\  
w<\n_K 
\fp \leq x
\\
f_\fp=1 }}\frac{f(\fp)}{\n_K\fp}\Bigg),
$$
if $f$ is  submultiplicative, and 
$$
E_{f}(x;W)=
\sum_{\substack{
\n_K\fa  \leq x\\
\fa\in \cP_K^\circ  \text{ square-free} \\
\gcd(\n_K \fa,W)=1
}}\frac{f(\fa)}{\n_K\fa},
$$
otherwise.

Suppose now that we are given irreducible binary forms $F_1,\dots,F_N\in \ZZ[x,y]$, which we assume to be pairwise coprime. 
Let $i\in \{1,\dots,N\}$. Suppose that $F_i$ has degree $d_i$
and that it is not proportional to $y$, so that  $b_i=F_i(1,0)$ is a non-zero integer.
It will be convenient to form the homogeneous polynomial
\begin{equation}\label{eq:burger}
\tilde F_i(x,y)=b_i^{d_i-1}F(b_i^{-1}x,y).
\end{equation}
This has integer coefficients and satisfies $\tilde F_i(1,0)=1$.
We  let $\theta_i$ be a root of the monic polynomial $\tilde F(x,1)$.
Then $\theta_i$ is an algebraic integer and we denote the  associated number field of degree $d_i$ by 
   $K_i=\QQ(\theta_i)$.
Moreover,  
\begin{equation}\label{eq:rancid}
N_{K_i/\QQ}(b_is-\theta_it)=\tilde F_i(b_is,t)=b_i^{d_i-1}F_i(s,t),
\end{equation}
for any $(s,t)\in \ZZ^2$.
(If  $b_i=0$, so that  $F_i(x,y)=c y$ for some non-zero $c \in \ZZ$,
we take  $\theta_i=-c$ and $K_i=\Q$ in this discussion.)
Our main result is a tight upper bound for averages of 
$f_{1,W}((b_1s-\theta_1 t))\dots f_{N,W}((b_Ns-\theta_N t)),$
over primitive vectors  $(s,t)\in \ZZ^2$, 
for general pseudomultiplicative functions $f_i\in \c{M}_{K_i}$ and suitably large values of $w$. 

Next, for any $k\in \NN$ and any polynomial $P\in \ZZ[x]$, we set 
$$
\rho_{P}(k)=\#\{x\bmod{k} : P(x)\equiv 0 \bmod{k}\}.
$$
We put 
\begin{equation}\label{eq:brain}
\bar\rho_i(k)=
\begin{cases}
\rho_{F_i(x,1)}(k) & \text{ if $F_i(1,0)\neq 0$,}\\
1 & \text{ if $F_i(1,0)= 0$},
\end{cases}
\end{equation}
and
\begin{equation}\label{eq:faggot}
h^*(k)=\prod_{p\mid k} \left(1-\frac{\bar\rho_1(p)+\dots +\bar\rho_N(p)}{p+1}\right)^{-1}.
\end{equation} 
To  any non-empty  bounded measurable  region
$\c{R} \subset \R^2$, 
we associate
\begin{equation}\label{eq:KR}
K_\c{R}=1+
\|\c{R}\|_{\infty}
+\partial(\c{R})
\log(1+\|\c{R}\|_\infty)
+\frac{\mathrm{vol}(\c{R})}{1+\|\c{R}\|_\infty}
,\end{equation}
where $\|\c{R}\|_\infty= \sup_{(x,y) \in \c{R}}\{ |x|,|y|\}$.
We say that such a region $\c{R}$ is {\em regular} if 
its boundary  is piecewise differentiable,  $\c{R}$ contains no 
zeros of 
$F_1\cdots F_N$  and 
 there exists
$c_1>0$ such that 
$\mathrm{vol}(\c{R})\geq K_\c{R}^{c_1}$.
Note that we then have 
$$
\mathrm{vol}(\c{R})\leq 4\|\c{R}\|_\infty^2\leq (1+\|\c{R}\|_\infty)^4\leq 
K_\c{R}^{4}.
$$
Bearing all of this in mind, we may now record our main result.

\begin{theorem}\label{t:NT}
Let $\c{R} \subset \R^2$ be a regular region,
let    $V=\mathrm{vol}(\c{R})$ 
and let $G \subset \Z^2$ be a lattice of full rank, with  determinant $q_G$ and  first successive minimum  $\lambda_G$.
Assume that  $q_G \leq V^{c_2}$ for some  $c_2>0$.
Let  $f_i \in \c{M}_{K_i}(A_i,B_i,\ve_i)$, for $1\leq i\leq N$ and let 
$$\ve_0= \max\bigg\{1+\frac{4}{c_1},\frac{4(5+3
\max\{\ve_1,\dots,\ve_N\})
}{c_1}\bigg\}
\Big(\sum_{i=1}^N d_i \ve_i\Big)
.$$ 
Then, for any  $\ve>0$ and 
$w>w_0(f_i,F_i,N)$, we have 
\begin{align*}
\sum_{(s,t) \in \Zp^2 \cap \c{R}\cap G}
\prod_{i=1}^N
f_{i,q_GW}((b_i s -\theta_i t))
\hspace{-0.08cm}
\ll~&
\frac{V}{(\log V)^N}
\frac{h_W^*(q_G)
}{q_G}
\prod_{i=1}^N
E_{f_i}(V;W)\\
&+\frac{K_{\c{R}}^{1+\ve_0+\ve}}{\lambda_G} 
,\end{align*}
where the implied constant depends at most on
$c_1,c_2,A_i,B_i,F_i, \ve, 
\ve_i, N,W$.
\end{theorem}

We
may compare this estimate with the principal result in work of la Bret\`eche--Tenenbaum \cite{moyennes}. Take $G=\ZZ^2$ and $d_1=\dots=d_N=1$. Then $b_is-\theta_it=F_i(s,t)$, 
for $1\leq i\leq N$. In this setting 
Theorem \ref{t:NT} can be deduced from the special case of \cite[Thm.~1.1]{moyennes}, in which all of the binary  forms are linear.

\begin{acknowledgements}
While working on this paper the first author was
supported by ERC grant \texttt{306457}.  
\end{acknowledgements}

\section{Technical results}

\subsection{Lattice point counting}

We will need general results about counting lattice points in an expanding region. 
Let $\mathbf{A}\in \mathrm{Mat}_{2\times 2}(\ZZ)$ be a non-singular upper triangular matrix and consider the lattice given by $G=\{\mathbf{Ay} : \y\in \ZZ^2\}$.  
Recall that  $G$ is said  {\em primitive} if the only integers $m$ fulfilling $G\subset m\ZZ^2$ are $m=\pm 1$.
We denote its {\em determinant} and first  {\em successive minimum}  by $\det(G)$ and $\lambda_G$, respectively.    Assume that $\cR\subset \RR^2$ is a regular region, in the sense of Theorem \ref{t:NT}.
Then, for any 
$\x_0\in \ZZ^2$ and $q\in \NN$ such that $\gcd(\det(G)\x_0,q)=1$, we  will require an asymptotic estimate for the counting function
$$
N(\cR)=\#\{\x\in \ZZp^2\cap \cR\cap G: \x\equiv \x_0\bmod{q}\}.
$$
The following estimate follows from work of Sofos \cite[Lemma~5.3]{sofos}.

\begin{lemma}\label{lem:5.3}
Assume that $G $ is primitive. Then 
\begin{align*}
N(\cR)=~&
\frac{\mathrm{vol}(\c{R})}{\zeta(2) \det(G) q^2}
\prod_{p \mid \det(G)} 
\l(1+\frac{1}{p}\r)^{-1}
\prod_{p \mid  q} 
\l(1-\frac{1}{p^2}\r)^{-1}\\
&+O\left(\frac{\tau(\det(G))K_\cR}{\lambda_G}
\right),
\end{align*}
where $K_\cR$ is given by \eqref{eq:KR} and the implied constant is absolute.
\end{lemma}

\subsection{Restriction to square-free support}

For a given number field $K/\QQ$ of degree $d$ and given $f\in \c{M}_K$, 
 it will sometimes be useful to bound 
sums of the shape 
$$
\sum_{\n_K \fa\leq x} \frac{f(\fa)}{\n_K \fa},
$$
by a sum restricted to square-free integral ideals supported away from ideals of small norm. This is encapsulated in the following result.

\begin{lemma}
\label{lem:knut}
Let $f\in \c{M}_K(A,B,\ve)$.  
Assume that  $f^\dagger$ is multiplicative and that
there exists  $M>0$ such that 
\[
|f^\dagger(\fp^\nu)-1|\leq \frac{M}{ \n_K \fp},
\] 
for all  prime ideals $\fp$ and $\nu\in \N$.
Assume that $W$ is given by \eqref{eq:WW}, with  $w>2(A+M)$.
Then 
\[
\sum_{\n_K \fa\leq x }
\frac{f(\fa_W)f^\dagger(\fa_W)}{\n_K\fa}
\ll_{A,M,W} 
\sum_{\substack{
\n_K \fb\leq x \\ 
\fb  \text{ square-free}
\\
\gcd(\n_K \fb,W)=1}}
\frac{f(\fb)}{\n_K \fb}
.\]
If $f$ is submultiplicative then the right hand side can be replaced by 
\[
\exp\l(\sum_{\substack{w<\n_K\fp \leq x }}\frac{f(\fp)}{\n_K\fp}\r)
.\]
\end{lemma}

\begin{proof}
The final part of the lemma is obvious. To see the first part we note that 
there is a unique factorisation
$\fa=\fq \fa_W$, where $\n_K\fq\mid W^\infty$.
Here, and throughout this paper, 
for $a,b\in \NN$
the notation 
$a\mid b^\infty$ is taken to mean 
that every prime divisor of $a$ is a divisor of $b$ as well.
Next, we decompose uniquely
$\fa_W=\fb \fc$
where
$\fb,\fc$ are coprime integral ideals such that
$\fb$  is square-free and $\fc$ is square-full.
The sum in the lemma is at most 
\[ 
\sum_{\substack{ \n_K\fq \leq x \\ 
\n_K\fq\mid W^\infty
 }} \frac{1}{\n_K \fq}
\sum_{\substack{\n_K\fb\leq x  \\ \fb \text{ square-free} \\ 
\gcd(\n_K\fb,W)=1}}
\frac{
f(\fb)f^\dagger(\fb)
}{\n_K\fb}
\sum_{\substack{
\n_K\fc \leq x \\
\fc  \text{ square-full}\\
\gcd(\n_K\fc,W)=1}}
\frac{
A^{\Omega_K(\fc)}
f^\dagger(\fc)
}{\n_K\fc}
.\]
For prime ideals with  $\n_K \fp > 2(A+M)$,
we have 
\[1+
\sum_{d=2}^\infty
\frac{
A^{\Omega_K(\fp^d)}
f^\dagger(\fp^d)
}{(\n_K \fp)^d}
\leq 1+
\sum_{d=2}^\infty
\frac{
A^d
}{(\n_K \fp)^d}
\left(1+\frac{M}{\n_K\fp}\right)
\leq 
\l(1-\frac{2A^2}{(\n_K\fp)^2}\r)^{-1}.
\]
Thus   the sum over $\fc$ converges absolutely.
Defining the multiplicative function $g:\c{I}_K\to \R$ via
$g(\fp^d)=
(f^\dagger(\fp)-1)\n_K\fp,$ 
for $d\in \N$, 
we clearly have
\[
f^\dagger(\fb)=\sum_{\fb=\fd\fe}\frac{g(\fd)}{\n_K\fd},
\]
with  $|g(\fd)|\leq M^{\Omega_K(\fd)}$,  for any square-free ideal $\fb$.
Therefore 
\[
\sum_{\substack{\n_K\fb\leq x  \\ \fb \text{ square-free} \\ 
\gcd(\n_K\fb,W)=1}}
\frac{
f(\fb)f^\dagger(\fb)
}{\n_K\fb}
=
\sum_{\substack{\n_K\fe\leq x  \\ \fe \text{ square-free} \\ 
\gcd(\n_K\fe,W)=1}}
\frac{
f(\fe)
}{\n_K \fe}
\sum_{\substack{\n_K\fd\leq x/(\n_K \fe) \\ \fd \text{ square-free} \\ 
\gcd(\n_K\fd,W)=1}}
\frac{
g(\fd)
}{(\n_K\fd)^2}
.\]
The sum over $\fd$ is absolutely convergent, whence
$$
\sum_{\n_K \fa\leq x }
\frac{f(\fa_W)f^\dagger(\fa_W)}{\n_K\fa}
\ll_{A,M,W} 
\sum_{\substack{\n_K\fe\leq x  \\ \fe \text{ square-free} \\ 
\gcd(\n_K\fe,W)=1}}
\frac{
f(\fe)
}{\n_K \fe}
\sum_{\substack{ \n_K\fq \leq x \\ 
\n_K\fq\mid W^\infty }} \frac{1}{\n_K \fq}.
$$
The inner sum over $\fq$ is at most
$
\prod_{\n_K\fp \mid W^d}
(1-\frac{1}{\n_K\fp})^{-1}
\ll_{A,M,W}1,
$
which thereby  completes the proof of the lemma.
\end{proof}

\subsection{The relevance of $\cP_K$}

Let $F\in \ZZ[x,y]$ be an irreducible non-zero binary form of degree $d$,
which  is not proportional to $y$.
In particular $b=F(1,0)$ is a non-zero integer.
We recall
from \eqref{eq:burger} the   associated binary form
$\tilde F(x,y)=b^{d-1}F(b^{-1}x,y)$,
with  integer coefficients and $\tilde F(1,0)=1$.  
We  let $\theta$ be a root of the polynomial $f(x)=\tilde F(x,1)$.
Then $\theta$ is an algebraic integer and   $K=\QQ(\theta)$ is a number field of degree $d$ over $\QQ$.   
It follows that 
$\ZZ[\theta]$ is an order
of $K$ with  discriminant 
$\Delta_\theta=
|\det(\sigma_i(\omega^j))|^2$, where
 $\sigma_1,\dots,\sigma_d:K \hookrightarrow\CC$
are  the associated 
embeddings.  As is well-known, 
we have 
\begin{equation}\label{eq:tesco}
\Delta_\theta=[\fo:\ZZ[\theta]]^2D_K,
\end{equation}
where $D_K$
is the  discriminant of $K$.
Recall the definition \eqref{eq:span}
of $\cP_K^\circ$
and define
\begin{equation}\label{eq:span}
\!\cP_K\!=\!
\left\{\fa\subset \cP_K^\circ: 
\fp_1\fp_2\mid \fa \Rightarrow
\n_K\fp_1\neq 
\n_K\fp_2
 \text{ or }
\fp_1= \fp_2
\right\}
,\end{equation}
which is
the subset of 
$\cP_K^\circ$
that is 
cut out by ideals 
divisible by at most one prime ideal above each rational prime.
The following result is crucial  in our analysis and will frequently allow us to restrict attention to ideals supported on $\cP_K$.

\begin{lemma}\label{lem:reuss}
Let $(s,t)\in \Zp^2$ such that  $F(s,t) \neq 0$.
Then 
$\fa\in \cP_K$ for any integral ideal $\fa\mid (bs-\theta t)$
such that $\gcd(\n_K\fa,2b \Delta_\theta)=1$.
\end{lemma}

\begin{proof}
Let $D=2b \Delta_\theta$ and let  $(s,t)\in \Zp^2$ such that  $F(s,t) \neq 0$. We form the integral ideal  
$\mathfrak{n}=(bs-\theta t).$
This has norm $\n_K \mathfrak{n}=|\tilde F(bs,t)|.$
Let $k\mid \tilde F(bs,t)$ with $\gcd(k,D)=1$.   
In particular   $\gcd(k,\Delta_\theta)=1$. 

Now let $p\mid k.$
Then  $p\nmid t$ since $\gcd(s,t)=1$ and $p\nmid b$.
We choose $\bar t\in \ZZ$ such that $t\bar t\equiv 1\bmod{p}.$
Let 
$\fp$ be any prime ideal such that $\fp\mid (p)$ and $\fp\mid \mathfrak{n}$.
Consider the group homomorphism
$$
\pi: \ZZ/p\ZZ\to (\ZZ[\theta]+\fp)/\fp,
$$ 
given by 
$m\mapsto m+\fp$.  Suppose that $\pi(m_1)=\pi(m_2)$ for $m_1,m_2\in \ZZ/p\ZZ$. Then $m_1-m_2\in \fp$, whence $\n_K\fp\mid (m_1-m_2)^d$. But this implies that $p\mid m_1-m_2$, since $\fp\mid (p)$,  and so $\pi$ is injective. Next suppose that $P(\theta) +\fp\in 
(\ZZ[\theta]+\fp)/\fp$, where $P(\theta)=\sum_{i}c_i \theta^i$ for $c_i\in \ZZ$. Since $\fp\mid \mathfrak{n}$ and $\fp\nmid \bar t$, we 
 have 
$bs\bar t- \theta\in \fp$. Thus  $P(\theta)-P(bs\bar t)\in \fp$. Now choose  $m\in \ZZ/p\ZZ$ such that 
$m\equiv P(bs\bar t)\bmod{p}$. It then follows that 
$\pi(m)=P(\theta)+\fp$.
Thus $\pi$ is surjective and so it is an isomorphism.  Hence
$
[\ZZ[\theta]+\fp:\fp]=p.
$
In view of \eqref{eq:tesco}, we also have 
$$
D_K[\fo:\ZZ[\theta]+\fp]^2[\ZZ[\theta]+\fp:\ZZ[\theta]]^2=\Delta_\theta.
$$
This implies that $p\nmid [\fo:\ZZ[\theta]+\fp]$.  Since $\n\fp$ is power of $p$, 
we readily  conclude that  
$$
\n\fp=[\fo:\ZZ[\theta]+\fp][\ZZ[\theta]+\fp:\fp]=p.
$$
This therefore establishes  
that $\fa \in \cP_K^\circ$.

To finish the proof   it remains to show that there are no 
distinct prime ideals 
$\fp_1,\fp_2$ with 
$\fp_1\fp_2\mid \fa$
and
$\n\fp_1=\n\fp_2$.
Suppose for a contradiction that there exist such primes $\fp_1\neq \fp_2$.
Letting $p=\n\fp_1=\n\fp_2$ and noting that 
 $p\nmid \Delta_\theta$, an application of 
Dedekind's theorem on factorisation of 
prime ideals
supplies us with  
 distinct $n_1,n_2\in \ZZ/p\ZZ$ such that 
$$
f(x)\equiv (x-n_1)(x-n_2)
\Upsilon(x) 
 \bmod{p},
$$
for a  polynomial $
\Upsilon  \in (\ZZ/p\ZZ)[x]$ of degree $[K:\Q]-2$, 
with $\fp_1=(p,\theta-n_1)$ and  $\fp_2=(p,\theta-n_2)$.
 We conclude from this that 
$bs\bar t-\theta \in \fp_1$ and $\theta-n_1\in \fp_1$, whence 
$bs\bar t-n_1\in \fp_1$. Similarly, we have
$bs\bar t-n_2\in \fp_2$.  But then it follows that  $p=\n\fp_1\mid bs\bar t-n_1$ and
 $p=\n\fp_2\mid bs\bar t-n_2$.
This implies that $n_1\equiv n_2\bmod{p}$, which is a contradiction. 
\end{proof}

We close this section with an observation about the condition 
$\mathfrak{a}
\mid (b s-\theta t)$ that appears in Lemma \ref{lem:reuss}.

\begin{lemma}\label{lem:honey}
Let $\fa\in \cP_K$ 
such that $\gcd(\n_K\fa,D_K)=1$.  
Then there exists $k
=k(\fa)
\in \ZZ$ such that
$$
\fa\mid (bs-\theta t) \Leftrightarrow 
bs\equiv k t\bmod{\n_K \fa}
$$
for all $(s,t)\in \ZZ^2$,
\end{lemma}

\begin{proof}
It suffices
to check this when 
$\fa=\fp^a$, for some $\fp\in \cP_K$ such that $p=\n_K\fp$ is unramified,
by the Chinese remainder theorem.
This is because   the definition of $\cP_K$ implies that for every rational unramified prime $p$
there is at most one prime ideal $\fp$ above $p$ such that $\fp \mid \fa$. To continue with the proof we note that 
since  $\fp\in \cP_K$,  there exists 
$k'\in \ZZ$ satisfying
$k'\equiv \theta \bmod{\fp}$, whence there exists $k\in \ZZ$ such that
$k\equiv \theta \bmod{\fp^a}$.
Therefore 
\begin{align*}
b s \equiv \theta t\bmod{\fp^a} ~\Leftrightarrow~
b s-k t \in \Z \cap \fp^a.
\end{align*}
We claim that the latter condition is equivalent to
$b s\equiv k t  \bmod{\n_K \fp^a}$. The reverse implication is obvious since $\n_K\fa\in \fa$ for any integral ideal $\fa$. 
The forward implication follows on noting that $\nu_p(bs-kt)\geq \nu_{\fp}((bs-kt))\geq a$.
\end{proof}

\section{The main argument}

This section is devoted to the proof of Theorem \ref{t:NT}, following
an approach that is inspired by  work of Shiu \cite{shiu}.
We assume familiarity with the notation introduced in \S \ref{s:intro}.
Since $F_1,\ldots ,F_N$ 
are pairwise coprime it follows that the resultants $\res(F_i,F_j)$ are all non-zero integers 
for $i\neq j$. 
Along the way, at certain stages of the argument,  we will need to enlarge the size of $W$ in \eqref{eq:WW}. For now we assume 
that $w>\max_{i\neq j} \{|D_i|, |\res(F_i,F_j)|\}$, where  
$D_i=2b_i \Delta_{\theta_i}$
and $\res$ denotes the resultant of two polynomials.
We let $\n_i=\n_{K_i}$ and write $F=\prod_{i=1}^NF_i$.
Let 
\begin{equation}
\label{eq:z}
z=V^\omega,
\end{equation}
where $V=\vol(\c{R})$, 
for a  small constant 
$\omega>0$  
that will be chosen in due course.
(In particular, it will need to be sufficiently 
small in terms of $\ve_1,\dots, \ve_N$.) 
For each $(s,t)\in \ZZp^2\cap \c{R}\cap G$, it follows from \eqref{eq:rancid} that we have a factorisation
\[
\prod_{i=1}^N
|\n_i(b_is-\theta_i t)_{q_G W}|
=\prod_{i=1}^N
|F_i(s,t)|_{q_G W}
=
|F(s,t)|_{q_G W}
=p_1^{\alpha_1}
\cdots p_l^{\alpha_l},
\] 
with $w<p_1<\cdots<p_l$.
We define $a_{s,t}$ to be the greatest integer of the form
$p_1^{\alpha_1}
\cdots p_j^{\alpha_j}$ which is bounded by  $z$ and we 
define $b_{s,t}=F(s,t)_{q_GW}/a_{s,t}$.
We have $\gcd(a_{s,t},b_{s,t})=1$ and 
$P^-(b_{s,t})>P^+(a_{s,t})$.
Our lower bound for $w$ ensures that 
\[
\gcd(\n_i \fa_i,\n_j \fa_j)=
1,
\]
for any $\fa_i\mid (b_is-\theta_it)_{q_GW}$ and 
$\fa_j\mid (b_js-\theta_jt)_{q_GW}$, 
with  $i\neq j$.
Lemma \ref{lem:reuss} implies that $\fa_i\in \cP_i=\cP_{K_i}$, for $1\leq i\leq N$.

The sum 
in which we are interested,
$$
\sum_{(s,t) \in \Zp^2 \cap \c{R}\cap G}
\prod_{i=1}^N
f_{i,q_GW}((b_i s -\theta_i t)),
$$
 is  sorted
into four distinct contributions
$E^{(I)}(\c{R}), \dots, E^{(IV)}(\c{R}).
$ 
For an appropriate small parameter $\eta>0$, 
these  sums are determined by which of the following attributes are satisfied by $(s,t)$:
\begin{itemize}
\item[({I})] $P^-(b_{s,t}) \geq z^{\frac{\eta}{2}}$;
\item[({II})] 
 $P^-(b_{s,t}) < z^{\frac{\eta}{2}}$ and 
 $a_{s,t}\leq z^{1-\eta}$;
\item[({III})] 
$P^-(b_{s,t})\leq \log z \log \log z$ and 
$z^{1-\eta}<a_{s,t}\leq z$;
\item[({IV})] 
$\log z \log \log z<
P^-(b_{s,t})<z^\frac{\eta}{2}$ and 
$z^{1-\eta}<a_{s,t}\leq z$.
\end{itemize}
In what follows, we will allow all of our implied constants to depend on 
$c_1,c_2,A_i,B_i, F_i, \ve, \ve_i,N,W$,  
as in the statement of Theorem \ref{t:NT}, as well as on $\omega$ and $\eta$, whose values will be indicated during the course of the proof.
Any further dependence will be indicated by an explicit subscript.

We let $\Omega_i=\Omega_{K_i}$ 
be the number of prime ideal divisors (counted with multiplicity)
and note that 
$\Omega_i(\fa)=\Omega(\n_i\fa)$ when $\fa\in \cP_i$. 
For given $(s,t)$, the choice of  $a_{s,t},b_{s,t}$ that we have made uniquely determines
coprime ideals 
$\fa_{s,t}^{(i)},
\fb_{s,t}^{(i)}
\subset \fo_i$, 
with  $(b_is-\theta_it)_{q_GW}=\fa_{s,t}^{(i)}\fb_{s,t}^{(i)}$, such that 
\[
\prod_{i=1}^N\n_i\fa_{s,t}^{(i)}=a_{s,t}
\quad 
\text{ and }
\quad
\Omega_i(\fa_{s,t}^{(i)})=
\Omega(\n_i\fa_{s,t}^{(i)}).
\]
In particular we emphasise that  $\fa_{s,t}^{(i)},\fb_{s,t}^{(i)}$
are supported on prime ideals whose norms are coprime to  $q_GW$.
We now have everything in place to start estimating the various contributions. 
Our main tools will be the geometry of numbers and the fundamental lemma of sieve theory. 

\subsection*{Case I}
We begin by considering  the case $P^-(b_{s,t}) \geq z^{\frac{\eta}{2}}$.
Recalling that $f_i\in \c{M}_{K_i}$, 
we have 
$
f_{i,q_GW}((b_is-\theta_i t))
\leq
f_i(\fa_{s,t}^{(i)})
A_i^{\Omega_{i}(\fb_{s,t}^{(i)})}, 
$
by the coprimality of $\fa_{s,t}^{(i)},\fb_{s,t}^{(i)}$.
Hence
\begin{equation}\begin{split}\label{eq:liver}
E^{(I)}(\c{R})
&=
\sum_{
\substack{
(s,t) \in \Zp^2 \cap \c{R}\cap G\\
P^-(b_{s,t})\geq z^{\frac{\eta}{2}}
}}
\prod_{i=1}^N f_{i,q_GW}((b_i s-\theta_i t))
\\
&\ll
\sum_{\substack{
\fa_i \in \c{P}_i
\\
\gcd(\n_i \fa_i,q_GW\n_j\fa_j)=1
\\
\prod_{i=1}^N\n_i\fa_{i}
\leq z
}}
\hspace{-0.4cm}
\c{U}(\fa_1,\dots,\fa_N)
\prod_{i=1}^N f_i(\fa_i),
\end{split}\end{equation}
where 
\[
\c{U}(\fa_1,\dots,\fa_N)
=
\hspace{-0.5cm}
\sum_{\substack{
(s,t) \in \Zp^2 \cap \c{R}\cap G\\
(b_is-\theta_i t)_{q_GW}
\in \cP_i\\
\fa_i \mid (b_is-\theta_i t)
\\
(\fa_i,(b_is-\theta_i t)/\fa_i)_i=1\\
p\mid F(s,t)_{q_GW}/\prod_{i=1}^N
\n_i \fa_{i} \Rightarrow p\geq z^\frac{\eta}{2} 
}}
\hspace{-0.5cm} 
\prod_{i=1}^N 
A^{\Omega_{i}\l((b_i s-\theta_i t)_{q_GW}/ {\fa_i}\r)}
.\]
Here, the condition 
$
\prod_{i=1}^N
\n_i \fa_{i} 
\leq z
$ comes from the fact that 
$a_{s,t} \leq z$. Moreover,
we write $(\fa,\fb)_i=1$ if and only if  the ideals $\fa,\fb\subset \fo_i$ are coprime.

Defining $b$ through
$b\prod_{i=1}^N
\n_i \fa_{i} =\prod_{p\nmid q_GW}p^{\nu_p(F(s,t))}$,
we see that 
\[
(z^\frac{\eta}{2})^{\Omega(b)}
\leq P^-(b)^{\Omega(b)}
\leq |b|
\leq
|F(s,t)|
\ll \|\c{R}\|_\infty^{\deg(F)}
\leq K_{\c{R}}^{\deg(F)}.
\]
In view of \eqref{eq:z} and the inequality  $K_\c{R}^{c_1} \leq V$ that is assumed in Theorem~\ref{t:NT}, this shows that 
$\Omega(b)\ll 1$.
Noting that
\[\sum_{i=1}^N
\Omega_{i}\l(\frac{(b_i s-\theta_i t)_{q_GW}}{ \fa_i}\r)
=\Omega(b),\]
we may therefore conclude  that 
\begin{equation}
\label{eq:upper-u}
\c{U}(\fa_1,\dots,\fa_N)\ll
\c{U}_\frac{\eta}{2}(\fa_1,\dots,\fa_N),
\end{equation}
where for any $\gamma>0$ we define 
$\c{U}_\gamma(\fa_1,\dots,\fa_N)$ to be the cardinality of $(s,t)$ 
appearing in the definition of 
$\c{U}(\fa_1,\dots,\fa_N)$, with the lower bound
$z^\frac{\eta}{2}$ replaced by $z^\gamma$.
Our next  concern is with an upper bound for this quantity.

Before revealing our estimate for 
$\c{U}_\gamma(\fa_1,\dots,\fa_N)$,
recall the definition of $h^*$  from \eqref{eq:faggot}
and set
\begin{equation}\label{eq:faggot'}
h^\dagger(k)=\prod_{p\mid k} \left(1-\frac{d_1+\dots+d_N}{1+p}\right)^{-1}.
\end{equation} 
Then we have the following result.

\begin{lemma}
\label{lem:salami}
Let $\delta,\gamma>0$ and let  $\fa_i\in \cP_i$, for $1\leq i\leq N$, with 
\[
\gcd(\n_i \fa_i,q_GW\n_j\fa_j)=1\quad\text{ and }
\quad
\prod_{i=1}^N
\n_i\fa_{i} \leq z.
\] 
Then 
\begin{align*}
\c{U}_\gamma(\fa_1,\dots,\fa_N)
\hspace{-0.08cm}
\ll_\delta
\hspace{-0.08cm}
\frac{V}{\gamma^{N}(\log z)^{N}}
\frac{h_W^*(q_G)}{q_G}
\prod_{i=1}^N 
\frac{
h^\dagger(\n_i\fa_i)}{\n_i\fa_i}
+\frac{K_\c{R} z^{2\gamma+\delta}}{\lambda_G},
\end{align*} 
uniformly in $\gamma$.
\end{lemma}

We defer the proof of this result, temporarily, and show how it can be used to complete the treatment of $E^{(I)}(\c{R})$,
via \eqref{eq:liver} and \eqref{eq:upper-u}.
We  apply Lemma~\ref{lem:salami}  with $\gamma=\frac{\eta}{2}$. 
We also note that since $f_i\in \c{M}_{K_i}(A_i,B_i,\ve_i)$, we have 
\[\prod_{i=1}^Nf_i(\fa_i)
\ll
\prod_{i=1}^N \n_i(\fa_i)^{\ve_i}
\leq 
\prod_{i=1}^N \n_i(\fa_i)^{\hat \ve},
\]
where $\hat \ve=\max \{\ve_1,\dots,\ve_N\}$.
The overall contribution from the second term is therefore  found to be
\begin{align*} 
&\ll_\delta
\frac{K_\c{R}z^
{\eta+\delta+\hat\ve} 
}{\lambda_G}
\#\left\{ \fa_1,\dots,\fa_N: 
\n_1\fa_{1}
\cdots \n_N\fa_{N}
\leq z\right\}\\
&\ll_\delta
\frac{K_\c{R}z^{1+\eta+2\delta+\hat \ve}}{\lambda_G}\\
&
\leq \frac{K_\c{R}^{1+4\omega(1+\eta+2\delta+\hat \ve)}}{\lambda_G},
\end{align*}
where we used the bound 
$V\leq K_\c{R}^4$, 
as well as $z=V^\omega$.
In view of \eqref{eq:liver} and \eqref{eq:upper-u}, the first term in Lemma \ref{lem:salami} makes the overall 
contribution 
\begin{align*}
\ll_{\delta}~&
\frac{V}{(\log z)^{N}} 
\frac{h_W^*(q_G)}{q_G}
\prod_{i=1}^N
\sum_{\substack{
\fa_i \in \cP_i
\\
\gcd(\n_i \fa_i,q_G W)=1
\\
\n_i\fa_{i}
\leq z
}}
\hspace{-0.2cm}
\frac{f_i(\fa_i) h^\dagger(\n_i\fa_i)}{\n_i\fa_i}.
\end{align*}
Since  $\fa_i\in \cP_i$,  \eqref{eq:faggot'} implies that 
\begin{align*} 
h^\dagger(\n_i\fa_i)
\leq \prod_{\fp \mid \fa_i}\l(1-\frac{d_1+\dots+d_N}{1+\n_i\fp}\r)
^{-1}
=
h_i^\ddagger(\fa_i),
\end{align*}
say, where we recall that $d_i=[K_i:\QQ]$.
We enlarge $w$ in order  to use Lemma~\ref{lem:knut},
and thereby  obtain the overall contribution
$$
\ll_{\delta}
\frac{V}{(\log z)^N}
\frac{h_W^*(q_G)}{q_G}
\prod_{i=1}^N
\sum_{\substack{
\fa_i \in \cP_i
\\
\gcd(\n_i \fa_i,q_G W)=1
\\
\fa_i \ \text{square-free}
\\
\n_i\fa_{i}
\leq z
}}\frac{f_i(\fa_i)}{\n_i\fa_i}.
$$
We have therefore proved that for every $\delta>0$ we have the bound 
\begin{equation}
\label{eq:ee11}
E^{(I)}(\c{R})
\ll_{\delta}
\frac{V}{(\log z)^N}
\frac{h_W^*(q_G)}{q_G}
\prod_{i=1}^N
\sum_{\substack{
\fa_i \in \cP_i
\\
\gcd(\n_i \fa_i,q_G W)=1
\\
\fa_i \ \text{square-free}
\\
\n_i\fa_{i}
\leq z
}}\frac{f_i(\fa_i)}{\n_i\fa_i}
+ \frac{K_\c{R}^{1+4\omega(1+\eta+2\delta+\hat \ve)}}{\lambda_G} 
,\end{equation}
where we recall that  $\hat \ve=\max \{\ve_1,\dots,\ve_N\}$. 

\begin{proof}[Proof of Lemma \ref{lem:salami}]
Let  $\fc_i  \in \cP_i$ be given, with $\gcd(\n_i \fc_i,q_GW\n_j\fc_j)=1$. 
Define the set 
\[
\Lambda(\fc_1,\dots,\fc_N)
=\left\{
(s,t) \in \Z^2\cap G :
b_is \equiv \theta_i t\bmod{\fc_i} ,  \text{ for $i=1,\dots, N$}
\right\}.
\]
Since
$\gcd(q_G,\prod_i\n_i \fc_i)=1$, it follows from Lemma \ref{lem:honey}
that 
this  defines a lattice in $\Z^2$ of rank $2$ and determinant 
$$
\mathrm{det}(\Lambda(\fc_1,\dots,\fc_N))=q_G\prod_{i=1}^N\n_i\fc_i.
$$

Write $P(z_0)=\prod_{p<z_0}p$, for any $z_0>0$, with the usual convention that 
$P(z_0)=1$ if $z_0<2$. This allows us to write
\begin{align*}
\c{U}_\gamma(\fa_1,\dots,\fa_N)
&\leq
\hspace{-0.3cm}
\sum_{\substack{
(s,t) \in S\\
(\fa_i,(b_is-\theta_i t)/\fa_i)_i=1\\
p\mid F(s,t)_{q_GW}/\prod_{i=1}^N
\n_i \fa_{i}   \Rightarrow p\geq z^\gamma
}}
\hspace{-0.3cm}
1
=
\hspace{-0.3cm}
\sum_{\substack{
(s,t) \in S\\
(\fa_i,(b_is-\theta_i t)/\fa_i)_i=1}}
\sum_{\substack{d\mid F(s,t)/\prod_{i=1}^N
\n_i\fa_i\\
\gcd(d,q_GW)=1
\\
d\mid P(z^\gamma)}}
\mu(d).
\end{align*}
where $S=\ZZp^2\cap \c{R}\cap \Lambda(\fa_1,\dots,\fa_N)$.
We shall use the fundamental lemma of sieve theory, as presented by Iwaniec and Kowalski~\cite[\S~6.4]{iwko}.
This  provides us with a sieve sequence $\lambda_d^+$ supported on square-free integers in the interval $[1,2z^\gamma]$, with 
$\lambda_1^+=1$ and $|\lambda_d^+|\leq 1$, 
such that 
\begin{align*}
\c{U}_\gamma(\fa_1,\dots,\fa_N)
&\leq
\sum_{\substack{
(s,t) \in S\\
(\fa_i,(b_is-\theta_i t)/\fa_i)_i=1}}
\sum_{\substack{d\mid F(s,t)/\prod_{i=1}^N
\n_i\fa_i\\
\gcd(d,q_GW)=1
\\
d\mid P(z^\gamma)}}
\lambda_d^+.
\end{align*}
Since $\gcd(a_{s,t},b_{s,t})=1$, we note that only $d$ coprime to $\prod_{i=1}^N
\n_i \fa_i$ appear in the inner sum. Interchanging the order of summation, we find that 
\begin{align*}
\c{U}_\gamma(\fa_1,\dots,\fa_N)
&\leq
\sum_{\substack{
\fe_i\mid \fa_i
}}
\mu_1(\fe_1)
\cdots \mu_N(\fe_N)
\hspace{-0.5cm}
\sum_{\substack{1\leq d \leq 2z^\gamma
\\ \gcd(d,q_GW)=1 \\
\gcd(d,\prod_{i=1}^N \n_i\fa_i)=1
\\
d\mid P(z^\gamma)}}
\hspace{-0.4cm}
\lambda^+_d
\sum_{\substack{
(s,t) \in\ZZp^2\cap\c{R}\cap \Lambda(\fa_1\fe_1,\dots,\fa_N\fe_N) \\
d\mid F(s,t)}}
\hspace{-0.4cm}
1\\
&=
\sum_{\substack{
\fe_i\mid \fa_i
}}
\mu_1(\fe_1)
\cdots \mu_N(\fe_N)
\hspace{-0.5cm}
\sum_{\substack{
d_1,\dots,d_N \in \N
\\
\gcd(d_i,q_GW\n_i\fa_i)
=\gcd(d_i,d_j \n_j\fa_j)=1 \\
d_1
\cdots d_N \leq 2z^\gamma
\\
d_1\cdots d_N \mid P(z^\gamma)
}}
\hspace{-0.5cm}
\lambda^+_{d_1 \cdots d_N}
S(\b{d}),
 \end{align*}
 where  
if  $d=d_1\cdots d_N$, then 
$$
S(\b{d})=
\sum_{\substack{(\sigma, \tau) \bmod{d}\\
\gcd(\sigma,\tau,d)=1\\
F_i(\sigma,\tau)\equiv 0 \bmod{d_i}}}
\sum_{\substack{
(s,t) \in \ZZp^2\cap\c{R}\cap \Lambda(\fa_1\fe_1,\dots,\fa_N\fe_N) \\
(s,t)\equiv (\sigma,\tau) \bmod{d}
}}
1.
$$
If  $F_i(x,y)=cy$ for some $i$,  then
the condition 
$b_i s\equiv \theta_i t \bmod{\fa_i \fe_i}$
should be  replaced by  $t \equiv 0 \bmod{\fa_i \fe_i}$. 

Recall the definition 
\eqref{eq:brain} of  $\bar\rho_i$ and  let 
\begin{equation*}
h(d)=\prod_{p\mid d}
\l(1+\frac{1}{p}\r)^{-1}.
\end{equation*}
The number of possible  $(\sigma,\tau)\bmod{d}$ is  equal to
$\phi(d) \bar\rho_1(d_1)\cdots \bar\rho_N(d_N)$.
In $S(\b{d})$ the inner sum can be estimated using  the geometry of numbers. 
Calling upon Lemma \ref{lem:5.3}, we deduce that 
\[
S(\b{d})=\frac{V}{\zeta(2)}
\frac{h(q_G)}{q_G}
\prod_{i=1}^N
\frac{\bar\rho_i(d_i)h(d_i)
h(\n_i\fa_i)}{d_i \n_i \fe_i \n_i \fa_i}
+O_\delta\l( \frac{K_{\c{R}}z^{\gamma+\frac{\delta}{2}}}{\lambda_G}
\r)
,\]
for any $\delta>0$. We emphasise that the implied constant in this estimate does  not depend on any of $\c{R},d_i,\fa_i$ or $\fe_i$. 
Since $|\lambda_d^+|\leq 1$
and $\tau_{K_i}(\fa_i) \ll_\delta( \n_i\fa_i)^{\frac{\delta}{2N}}$, we find that 
the  overall contribution to $\c{U}_\gamma (\fa_1,\dots,\fa_N)$ from the error term is 
$O_\delta(K_\c{R}z^{2\gamma+\delta}/\lambda_G)$, 
on summing trivially over $\fe_1,\dots,\fe_N$ and $d_1,\dots,d_N$.
This is plainly satisfactory for Lemma \ref{lem:salami}.

Turning to the contribution from the main term, we set
\[
g(d)=
\mathbf{1}\Big(d,q_GW\prod_{i=1}^N \n_i \fa_i\Big) 
\frac{h(d)}{d}
\sum_{\substack{
d_1
\cdots d_N=d
\\
\gcd(d_i,d_j)=1
}}
\prod_{i=1}^N
\bar\rho_i(d_i)
,\]
where $\mathbf{1}(d,a)=1$ if 
$\gcd(d,a)=1$ and  
$\mathbf{1}(d,a)=0$, otherwise.
Since $h(d)\leq 1$
and $\varphi_i(\fa_i)\leq \n_i\fa_i$, 
the main term contributes
\begin{align*}
&\ll 
\frac{V}{q_G}
\prod_{i=1}^N
\frac{1}{\n_i \fa_i}
\Big|\sum_{\substack{
1\leq d \leq 2z^\gamma
\\
d \mid P(z^{\gamma})
}}
\lambda_d^+ g(d)\Big|.
\end{align*}
We may clearly  assume without loss of generality that $w< 2z^{\max\{\gamma, \frac{\eta}{2}\}}$. 
For any prime $p\nmid W$, let 
$$
c_p=
1-\frac{h(p)}{p}\sum_{i=1}^N
\bar\rho_i(p)=1-\frac{\bar\rho_1(p)+\dots+\bar\rho_N(p)}{p+1}.
$$
Recalling that $\deg(F_i)=d_i$ for $1\leq i\leq N$, we see that 
$$
c_p\geq 1-\frac{d_1+\dots+d_N}{p+1},
$$
for $p\nmid W$.
Next, for $y\geq 0$, define
\[
\Pi(y)=\prod_{\substack{p<y\\ p\nmid W}}c_p, \quad 
\Pi_1=
\prod_{\substack{p\geq 2z^\gamma \\ p| \n_1\fa_1\cdots \n_N\fa_N}}c_p,
\quad \Pi_2=
\prod_{\substack{p\geq 2z^\gamma \\ p| q_G }}c_p
.\]
By the fundamental lemma of sieve theory  \cite[Lemma 6.3]{iwko}, 
we find that  
\begin{align*}
\sum_{\substack{
1\leq d \leq 2z^\gamma
\\
d \mid P(z^\gamma)
}}
\lambda^+_{d}g(d)
&\ll
\prod_{p<2z^\gamma}(1-g(p))
=
\Pi(2z^\gamma)
\Pi_1 \Pi_2
h_W^*\l(q_G\r)
\prod_{i=1}^N h^\dagger(\n_i\fa_i),
\end{align*}
in the notation of \eqref{eq:faggot} and
\eqref{eq:faggot'}.
It is clear that $\Pi_i\leq 1$ for $i=1,2$.
Noting that $\Pi(y)\ll (\log y)^{-N}$,
this therefore concludes the proof of the lemma.
\end{proof}

\subsection*{Cases II and III}

We now estimate
$E^{(II)}(\c{R})$ and 
$E^{(III)}(\c{R})$. 
For any  
$(s,t)\in \ZZp^2\cap \c{R}$,  we take the  trivial bound 
$$
\prod_{i=1}^N
f_{i,q_GW}((b_is-\theta_i t))
\ll 
\prod_{i=1}^N (\n_i(b_is-\theta_i t))^{\ve_i}
\ll 
\|\c{R}\|_\infty^{\sum_i d_i\ve_i}
\le
 K_\c{R}^{\sum_i d_i\ve_i}
.$$
In Case II
the relevant extra  constraints are
$P^-(b_{s,t}) < z^{\frac{\eta}{2}}$ and 
 $a_{s,t}\leq z^{1-\eta}$.
Let
$p=P^-(b_{s,t})< z^\frac{\eta}{2}$ and let $\nu$ be such that  
$p^\nu\| F(s,t)$. 
We must have  $p^\nu\geq z^{\eta}$, since otherwise
$z<p^\nu
a_{s,t}
<
z^{\eta}
z^{1-\eta}=z$
,
which is a contradiction.
For each prime $p\nmid q_GW$ with $p< z^\frac{\eta}{2}$, we  define
\[
l_p=\min\{
l \in \Z_{\geq 0}:p^l\geq z^{\eta}
\}
.\]
Clearly 
$
z^{\eta}
\leq p^{l_p}
<
z^{\frac{\eta}{2} l_p},$ 
whence $l_p\geq 2$ for every prime $p$.
Therefore 
\begin{equation}
\label{eq:lp}
\sum_{\substack{p<z^\frac{\eta}{2} \\ p\nmid q_GW}}
\frac{1}{p^{l_p}}
\leq
\sum_{\substack{p<z^\frac{\eta}{2} \\ p\nmid q_G W}}
\min\left\{
\frac{1}{z^{\eta}},
\frac{1}{p^{2}}
\right\}
\leq
\sum_{p\leq z^{\frac{\eta}{2}}} \frac{1}{z^{\eta}}  \leq  z^{-\frac{\eta}{2}}
.\end{equation}
The number of  elements $(s,t)$  satisfying the constraints of Case II is at most
\[
\sum_{i=1}^N
\sum_{\substack{p< z^\frac{\eta}{2} \\p\nmid  q_G W}}
\sum_{\substack{(s,t)\in \ZZp^2\cap \c{R}  \cap G 
\\
p^{l_p}\mid F_i(s,t)}}
1\ll 
\sum_{i=1}^N
\sum_{\substack{p< z^\frac{\eta}{2} \\p\nmid  q_G W}}
\bar\rho_i(p^{l_p})
\l(
\frac{h(q_G)}{q_G}
\frac{V}{p^{l_p}}
+\tau(q_G)
\frac{K_\c{R}}{\lambda_G}
\r).
\]
Here we have split the inner sum 
into $\bar\rho_i(p^{l_p})$
different lattices of the form
$\{(s,t) \in G:  s\equiv x t \bmod{p^{l_p}}\}$,
where $x$ ranges over solutions of the congruence $F_i(x,1)\equiv 0 \bmod{p^{l_p}}$, before applying 
Lemma \ref{lem:5.3}
with $q=1$.
Hensel's lemma implies that $\bar\rho_i(p^l)=\bar\rho_i(p)\leq d_i$ 
for each  prime $p\nmid W$ and $l \in \N$. 
Let $\delta>0$ be arbitrary.
Taking  $h(q_G)\leq 1$ and  $\tau(q_G)\ll_\delta q_G^{\frac{\delta}{8 c_2}}\leq V^{\frac{\delta}{8}}\le K_\c{R}^{ \frac{\delta}{2}}$, this therefore reveals that 
\begin{align*}
\sum_{i=1}^N
\sum_{\substack{p< z^\frac{\eta}{2} \\p\nmid q_GW}}
\sum_{\substack{(s,t)\in \ZZp^2\cap \c{R}  \cap G
\\
p^{l_p}\mid F_i(s,t)}}
1
\ll_\delta
\frac{V}{q_G}
\sum_{\substack{p< z^\frac{\eta}{2} \\p\nmid q_GW}}
\frac{1}{p^{l_p}}  
+\frac{K_{\c{R}}^{1+\frac{\delta}{2}}z^{\frac{\eta}{2}}}{\lambda_G}
&\ll_\delta
\frac{V}{q_G}
\frac{1}{z^{\frac{\eta}{2}}}
+\frac{K_{\c{R}}^{1+\frac{\delta}{2}}z^{\frac{\eta}{2}}}{\lambda_G},
\end{align*}
by \eqref{eq:lp}.
Recalling \eqref{eq:z}, we see that
$
z^{-\frac{\eta}{2}}
= V^{-\beta}
$,
with 
$\beta=\frac{\eta \omega}{2}$.
Likewise, 
$z^\frac{\eta}{2}=V^{\frac{\eta \omega}{2} }
\leq
K_{\c{R}}^{2 \eta\omega   }
$.
Noting that 
$$
K_\c{R}^{\sum_i d_i \ve_i}
\leq 
V^{\frac{1}{c_1} \sum_i d_i \ve_i}
,$$
we have therefore proved that for every $\delta>0$ we have the bound 
\begin{equation}
\label{eq:ee22}
E^{(II)}(\c{R})
\ll_{\delta}
\frac{V^{1-\frac{\eta \omega}{2}+\frac{1}{c_1} \sum_i d_i \ve_i}}{q_G} 
+ \frac{K_\c{R}^{1+\frac{\delta}{2}+2 \eta \omega  +\sum_i d_i \ve_i}}{\lambda_G}
.\end{equation}

We now turn to the contribution from Case III, for which the defining constraints are 
$
P^-(b_{s,t})\leq \log z \log \log z$
and $z^{1-\eta}<a_{s,t}\leq z$.
We assume that $w>\max_{i\neq j}|\res(F_i,F_j)|$ in the definition \eqref{eq:WW} of $W$. For any $(s,t)\in \ZZp^2$ it follows that the integer factors of 
$F_i(s,t)_W$ are necessarily coprime to the factors of $F_j(s,t)_W$ for all $i\neq j$.
Hence the number of elements $(s,t)$ 
satisfying the constraints of Case III is at most
\begin{align*}
\sum_{\substack{(s,t)\in \ZZp^2\cap \c{R} \cap G 
\\
P^-(b_{s,t})\leq  \log z \log \log z
\\
z^{1-\eta}<a_{s,t}\leq z}}
1
&\leq
\sum_{\substack{z^{1-\eta}<a\leq z
\\
\gcd(a,q_GW)=1
\\
P^+(a) \leq 
\log z
\log \log z
}}
\sum_{\substack{(s,t)\in \ZZp^2\cap \c{R} \cap G\\
a\mid F(s,t)}}1\\
&\leq
\sum_{\substack{z^{1-\eta}<a_1\cdots a_N\leq z
\\
\gcd(a_i,q_G Wa_j)=1
\\
P^+(a_i) \leq 
\log z
\log \log z
}}
\sum_{\substack{(s,t)\in \ZZp^2\cap \c{R} \cap G\\
a_i\mid F_i(s,t)}}1.
\end{align*}
As before, the final sum can be split into at most $\prod_{i=1}^N \bar\rho_i(a_i)=O_\delta(z^\delta)$
lattices, for any $\delta>0$, 
each of determinant
$q_G\prod_{i=1}^N a_i$.
Thus the right hand side is
\begin{align*}
&\ll_\delta z^{\delta} 
\sum_{\substack{z^{1-\eta}<a_1
\cdots a_N\leq z
\\
P^+(a_1\cdots a_N) \leq
\log z
\log \log z
}}
\l(
\frac{V}{q_Ga_1 
\cdots 
a_N}
+
\frac{K_\c{R}}{\lambda_G} 
\r)\\
&\ll_\delta z^{2\delta} 
\sum_{\substack{z^{1-\eta}<a\leq z
\\
P^+(a) \leq \log z \log \log z
}}
\l(
\frac{V}{q_Ga}
+
\frac{K_\c{R}}{\lambda_G} 
\r),  
\end{align*}
whence 
\cite[Lemma~1]{shiu} yields  
\begin{align*}
\sum_{\substack{(s,t)\in \ZZp^2\cap \c{R} \cap G
\\
P^-(b_{s,t})\leq \log z\log \log z\\
z^{1-\eta}<a_{s,t}\leq z}}
\hspace{-0,3cm}1 \ 
&\ll_{\delta}
z^{3\delta}
\l(
\frac{V}{q_Gz^{1-\eta}}+
\frac{K_\c{R}}{\lambda_G} 
\r).
\end{align*}
Hence 
we have shown
that for every $\delta>0$
one has 
\begin{equation}
\label{eq:ee33}
E^{(III)}(\c{R})
\ll_{\delta}
\frac{V^{1-(1-\eta)\omega +3\delta \omega +\frac{1}{c_1} \sum_i d_i \ve_i}}{q_G} 
+ \frac{K_\c{R}^{1+3 \delta \omega+\sum_i d_i \ve_i}}{\lambda_G}
.\end{equation}

\subsection*{Case IV}

The final case to consider is characterised by the constraints
$$
\log z 
\log \log z
<
P^-(b_{s,t})<z^\frac{\eta}{2} \quad \text{ and  } \quad
z^{1-\eta}<a_{s,t}\leq z.
$$
Arguing as in \eqref{eq:liver} in the treatment of Case I, we find that 
\begin{align*}
E^{(IV)}(\c{R})
&\ll
\sum_{\substack{
\fa_i \in \cP_i
\\
\gcd(\n_i \fa_i,q_GW\n_j\fa_j)=1
\\
z^{1-\eta}<\prod_{i=1}^N
\n_i\fa_{i}
\leq z
}}
\c{U}^{\dagger}(\fa_1,\dots,\fa_N)
\prod_{i=1}^N f_i(\fa_i),
\end{align*}
where $\c{U}^\dagger(\fa_1,\dots,\fa_N)$ is as in  the definition 
of $\c{U}(\fa_1,\dots,\fa_N)$ after \eqref{eq:liver}, but with 
the condition 
$P^-(F(s,t)_{q_GW}/\prod_{i=1}^N
\n_i \fa_{i})\geq z^\frac{\eta}{2}$ replaced by 
$$
\log z
\log \log z
< P^-\l(
\frac{F(s,t)_{q_GW}}{\prod_{i=1}^N
\n_i \fa_{i}}\r)< z^\frac{\eta}{2}.
$$
In particular, in view of the coprimality of  $a_{s,t}$ and $b_{s,t}$, we see that 
\[
P^+\left(\prod_{i=1}^N
\n_i\fa_{i}\right) 
<P^-\l(
\frac{F(s,t)_{q_GW}}{\prod_{i=1}^N
\n_i \fa_{i}}
\r)
.\]
We will find it convenient to enlarge the sum slightly, replacing the condition $\fa_i\in \cP_i$ by the condition that each $\fa_i$ belongs to the multiplicative span of degree $1$ prime ideals in $\fo_i$.

We may assume without loss of generality that $\frac{2}{\eta}\in \Z_{\geq 2}.$ Thus 
\[
(
\log z
\log \log z,
z^\frac{\eta}{2}
)
\subset
\bigcup_{k=\frac{2}{\eta}}^{k_0+1}
(z^{\frac{1}{k+1}},z^{\frac{1}{k}}]
,
\]
where $k_0=[\log z/\log(
\log z
\log \log z)]$
satisfies
$k_0\leq \log z/ \log \log z$.
Notice that for any integer $b$ satisfying
$\log b \ll \log z$ and $z^{\frac{1}{k+1}}<P^-(b)\leq z^{\frac{1}{k}}$ 
we must have 
$\Omega(b) \ll  k$.
Applying this with  $b=F(s,t)_{q_GW}/\prod_{i=1}^N
\n_i\fa_{i}$,
for any $(s,t)\in \ZZp^2\cap \c{R}$,
we deduce that 
\begin{align*}
\prod_{i=1}^N
A_i^{\Omega_i\big(\frac{(b_is-\theta_i t)_{q_GW}}{\fa_i}\big)}
&\leq
 \max_{1\leq i\leq N}A_i^{\Omega(b)}
\leq
A^k,
\end{align*}
for a suitable constant 
$A\gg \max_{1\leq i\leq N}A_i$, where  
 $A_i$ is the constant appearing in the definition of 
$\c{M}_{K_i}=\c{M}_{K_i}(A_i,B_i,\ve_i)$. 
Hence
$$
E^{(IV)}(\c{R})
\leq
\sum_{k=\frac{2}{\eta}}^{{k_0+1}}
A^k
\sum_{\substack{
\fa_i\in \cP_{K_i}^\circ
\\
\gcd(\n_i \fa_i,q_GW\n_j\fa_j)=1
\\
z^{1-\eta}<\prod_{i=1}^N
\n_i\fa_{i}
\leq z\\
P^+(\prod_{i=1}^N
\n_i\fa_{i})
< z^{\frac{1}{k}}
}}
\c{U}_{\frac{1}{k+1}}(\fa_1,\dots,\fa_N)
\prod_{i=1}^N
f_i(\fa_i) ,
$$
in the notation of  Lemma \ref{lem:salami}, which we now use  to estimate
$\c{U}_{\frac{1}{k+1}}(\fa_1,\dots,\fa_N)$.

The overall contribution from the second  term is 
$$
\ll_\delta
\frac{K_{\c{R}}}{\lambda_G}
\sum_{k=\frac{2}{\eta}}^{{k_0+1}}
A^k z^{1+\frac{2}{k+1}+2\delta}
\leq 
\frac{K_{\c{R}}z^{\frac{5}{3}+2\delta}}{\lambda_G}
\sum_{k=\frac{2}{\eta}}^{{k_0+1}}
A^k 
\ll_{\delta}
\frac{K_{\c{R}}z^{\frac{5}{3}+3\delta}}{\lambda_G}
\leq \frac{K_\c{R}^{1+ 4\omega(\frac{5}{3}+3\delta)}}{\lambda_G}
,$$
since   $2\leq 2/\eta \leq k \leq k_0\ll \log z/\log\log z$ and $z=V^\omega \leq  K_\c{R}^{4\omega}$. 

It remains to consider the   contribution from the main term in Lemma \ref{lem:salami}. This is 
\begin{equation}\label{eq:spam}
\ll \frac{V}{(\log z)^N}
\frac{
h_W^*(q_G)}{q_G}
\sum_{k=\frac{2}{\eta}}^{{k_0+1}}
A^k 
(k+1)^{N}
E(z^{1-\eta},z^{\frac{1}{k}}),
\end{equation}
where 
$$
E(S,T)=
\sum_{\substack{
\fa_i\in \cP_{K_i}^\circ
\\
\gcd(\n_i \fa_i,W\n_j\fa_j)=1
\\
\prod_{i=1}^N
\n_i\fa_{i}>S\\
P^+(\prod_{i=1}^N\n_i\fa_{i})
< T 
}}
\prod_{i=1}^N
\frac{f_i(\fa_i)h_i^\dagger(\fa_i)}{\n_i\fa_i}
.$$
Note that we have dropped the condition $\gcd(\prod_{i=1}^N \n_i\fa_i,q_G)=1$.

Let us define the multiplicative function
$u:\N\to \R_{\geq 0}$
via
\begin{equation}
\label{eq:deafinition}
u(a)=\sum_{\substack{
\fa_i\in \cP_{K_i}^\circ
\\\gcd(\n_i\fa_i,\n_j\fa_j)=1\\ \prod_{i=1}^N \n_i \fa_i=a}}
\prod_{i=1}^N
f_i(\fa_i)h_i^\dagger(\fa_i)
.\end{equation}
Note that 
\begin{equation}
\label{eq:c1}
u(p^k)=\sum_{i=1}^N
\sum_{\substack{\fp_i \subset \fo_i \text{ prime}  \\ \n_i \fp_i=p}}
f_i(\fp_i^k)
h_i^\dagger(\fp_i^k) \leq C^k
,\end{equation}
 for an appropriate constant $C>1$  depending on $A_i,d_i$ and $N$.
We may therefore write
\[
E(S,T)=
\sum_{\substack{
\gcd(a,W)=1
\\
a>S\\
P^+(a)
< T 
}}
\frac{u(a)}{a}.
\] 
Drawing inspiration from the proof of~\cite[Lemma 2]{NT},
we shall find an  upper bound for $E(S,T)$ in terms of partial sums involving $u(a)$. This is the object of the following result.

\begin{lemma}\label{lem:foul mouth}
Assume that $T>\eul^{\frac{C}{10}}$ and 
let $\kappa\in (\frac{1}{10}, C^{-1}\log T)$.
Then 
$$
E(S,T)
\ll_\kappa
\eul^{-\kappa \frac{\log S}{\log T}} 
\sum_{\substack{
\gcd(b,W)=1
\\
 b\leq T 
}}
\frac{u(b)}{b}
.
$$
\end{lemma}

Taking this result on faith for the moment, we return to~\eqref{eq:spam}
and apply it with $\kappa$ satisfying 
$\eul^{\kappa (1-\eta)}>2A$. 
(Note that the implied constant in Lemma \ref{lem:foul mouth} depends on $\kappa$ and so the choice $\kappa=(\log 2A)/(1-\eta)+1$ is acceptable.)
This produces the overall contribution
\begin{align*}
&\ll
\frac{V}{(\log z)^N}\frac{h_W^*(q_G)}{q_G}
\sum_{k=\frac{2}{\eta}}^{{k_0+1}}
\frac{A^k (k+1)^{N}}{\eul^{\kappa k(1-\eta)}}   
\sum_{\substack{
\gcd(b,W)=1
\\
 b\leq z^{\frac{1}{k}}
}}
\frac{u(b)}{b}
\\
&\ll
\frac{V}{(\log z)^N}
\frac{h_W^*(q_G)}{q_G}
\sum_{k=\frac{2}{\eta}}^{{k_0+1}}
\frac{(k+1)^{N}}{2^k}
\sum_{\substack{
\gcd(b,W)=1
\\
 b\leq z
}}
\frac{u(b)}{b}
\\
&\ll
\frac{V}{(\log z)^N}
\frac{h_W^*(q_G)}{q_G}
\sum_{\substack{
\gcd(b,W)=1
\\
 b\leq z
}}
\frac{u(b)}{b}
.\end{align*}
Recalling~\eqref{eq:deafinition} and enlarging $w$ to enable the use of 
Lemma~\ref{lem:knut},
shows that the last quantity is  
$$
\ll
\frac{V}{(\log z)^N}
\frac{h_W^*(q_G)}{q_G}
\prod_{i=1}^N
\sum_{\substack{
\n_i\fa_i  \leq z\\
\fa_i\in \cP_{K_i}^\circ \text{ square-free}
\\
\gcd(\n_i \fa_i,W)=1
}}\frac{f_i(\fa_i)}{\n_i\fa_i},
$$
which shows that for every $\delta>0$ we have
\begin{equation}
\label{eq:ee44}
E^{(IV)}(\c{R})
\ll_{\delta}
\frac{V}{(\log z)^N}
\frac{h_W^*(q_G)}{q_G}
\prod_{i=1}^N
\sum_{\substack{
\n_i\fa  \leq z\\
\fa_i\in \cP_{K_i}^\circ \text{ square-free}
\\
\gcd(\n_i \fa,W)=1
}}\frac{f_i(\fa_i)}{\n_i\fa_i}
+ \frac{K_\c{R}^{1+ 4\omega(\frac{5}{3}+3\delta)}}{\lambda_G}
.\end{equation}

\begin{proof}[Proof of Lemma \ref{lem:foul mouth}]
Let $\beta=\frac{\kappa}{\log T}$.
Then 
\[
E(S,T)\leq 
S^{-\beta}
\sum_{\substack{
\gcd(a,W)=1
\\
P^+(a)
< T 
}}
\frac{u(a)}{a}
a^\beta
.\]
Define the multiplicative function $\psi_\beta$ via
$
a^\beta=\sum_{c\mid a}\psi_\beta(c),
$
for $a\in \N$.
We observe that  
$\psi_\beta(p^k)=p^{\beta k}-p^{\beta (k-1)}$, for any $k \in \N$,
whence $0<\psi_\beta(a)<a^\beta$ for all $a \in \N$.
We now have 
\[
E(S,T)\leq 
S^{-\beta}
\sum_{\substack{\gcd(c,W)=1\\ P^+(c)<T}}
\frac{\psi_\beta(c)}{c}
\sum_{\substack{
\gcd(d,W)=1
\\
P^+(d)
< T 
}}
\frac{u(cd)}{d}
.\]
Writing $d=j d'$, with $\gcd(d',c)=1$ and $j \mid c^\infty$,
shows that
\[
E(S,T)\leq S^{-\beta}
\sum_{\substack{
\gcd(d',W)=1
\\
P^+(d')
< T 
}}
\frac{u(d')}{d'}
\sum_{\substack{\gcd(c,d'W)=1\\ P^+(c)<T}}
\sum_{j \mid c^\infty}
\frac{\psi_\beta(c)u(c j)}{c j}
.\]
After possibly enlarging $w$, it follows from \eqref{eq:c1} that 
the sum over $c$ is
\begin{align*}
\leq 
\prod_{\substack { p < T\\p\nmid d'W}}
\l(1+
\sum_{\substack{k \geq 1 \\ j \geq 0}}
\frac{\psi_\beta(p^k)
u(p^{k+j})}{p^{k+j}}
\r)
&\leq 
\prod_{\substack {p < T\\p\nmid d'W}}
\l(1+
\sum_{\substack{k \geq 1 \\ j \geq 0}}
\frac{(p^{\beta k}-p^{\beta (k-1)})C^{k+j}}{p^{k+j}}
\r)\\
&\leq 
\exp\l(
O\l(
\sum_{\substack{p<T\\p\nmid d'W}}\frac{p^{\beta}-1}{p}
\r)
\r)
.
\end{align*}
Writing $p^\beta=\exp(\frac{\kappa \log p}{\log T})=1+O(\frac{\kappa \log p}{\log T}),$
this is found to be at most
\[
 \exp\l(O\l(\frac{\kappa}{\log T}\sum_{p \leq T}\frac{\log p}{p}
\r)\r)
\ll_\kappa 1.
\]
Our argument so far shows that 
\begin{equation}
\label{eq:27}
E(S,T)
\ll_\kappa	
\eul^{-\kappa\frac{\log S}{\log T}}   
\sum_{\substack{
\gcd(d,W)=1
\\
P^+(d)
< T 
}}
\frac{u(d)}{d}
.\end{equation}

Let 
 $\xi \in (0,1)$.  
Observe that each $d$ with $P^+(d)<T$ can be written
uniquely in the form $d=d_-d_+$ for coprime $d_-, d_+\in \N$ such that 
$P^+(d_-)\leq T^{\xi}$
and $P^-(d_+) > T^{\xi}$. We clearly have  $P^+(d_+)<T$.
Thus 
\[
\sum_{\substack{
\gcd(d,W)=1
\\
P^+(d)
< T 
}}
\frac{u(d)}{d}
\leq
\sum_{\substack{
\gcd(d_-,W)=1
\\
P^+(d_-)
\leq T^{\xi}
}}
\frac{u(d_-)}{d_-}
\sum_{\substack{
\gcd(d_+,W)=1
\\
P^-(d_+)
> T^{\xi} 
\\
P^+(d_+)
< T 
}}
\frac{u(d_+)}{d_+}.
\] 
By~\eqref{eq:c1}, the inner sum is at most
$\prod_{T^{\xi} < p < T } (1+\frac{1}{p})^{2C}
\ll_{C,\xi}1.$
Thus, once combined with ~\eqref{eq:27}, we deduce that
\[
E(S,T)
\ll_\kappa
\eul^{-\kappa\frac{\log S}{\log T}}  
\sum_{\substack{
\gcd(d_-,W)=1
\\
P^+(d_-)
< T^{\xi}
}}
\frac{u(d_-)}{d_-}.
\]
In order to complete the proof of the lemma, it  remains to show that the 
$$
\sum_{\substack{
\gcd(d_-,W)=1
\\
P^+(d_-)
\leq T^{\xi}
}}
\frac{u(d_-)}{d_-}
\ll
\sum_{\substack{
\gcd(d,W)=1
\\
d<T
}}
\frac{u(d)}{d}.
$$
This is trivial when  $T^{\xi}<2$. Suppose now that $T^{\xi}>2$. 
Taking $\kappa=2$ and 
$(T,T^{\xi})$
in place of $(S,T)$, it follows from~\eqref{eq:27} that 
$$
\sum_{\substack{
\gcd(d_-,W)=1
\\
d_->T\\
P^+(d_-)
< T^{\xi}
}}
\frac{u(d_-)}{d_-}\ll
\eul^{-\frac{2}{\xi}}
\sum_{\substack{
\gcd(d_-,W)=1
\\
P^+(d_-)
< T^{\xi}
}}
\frac{u(d_-)}{d_-}.
$$
Taking  $\xi$
 suitably small, we conclude that 
\begin{align*}
\sum_{\substack{
\gcd(d_-,W)=1
\\
P^+(d_-)
< T^{\xi}
}}
\frac{u(d_-)}{d_-}
&=
\sum_{\substack{
\gcd(d_-,W)=1
\\
P^+(d_-)
< T^{\xi}
\\
d_-\leq T
}}
\frac{u(d_-)}{d_-}
+
\sum_{\substack{
\gcd(d_-,W)=1
\\
P^+(d_-)
< T^{\xi}
\\
d_-> T
}}
\frac{u(d_-)}{d_-}
\\
&
\leq
\sum_{\substack{
\gcd(d_-,W)=1
\\
 d_-\leq T 
}}
\frac{u(d_-)}{d_-}
+\frac{1}{2}
\sum_{\substack{
\gcd(d_-,W)=1
\\
P^+(d_-)
< T^{\xi}
}}
\frac{u(d_-)}{d_-}
,\end{align*}
so that 
\[\sum_{\substack{
\gcd(d_-,W)=1
\\
P^+(d_-)
< T^{\xi}
}}
\frac{u(d_-)}{d_-}
\leq
2\sum_{\substack{
\gcd(d_-,W)=1
\\
 d_-\leq T 
}}
\frac{u(d_-)}{d_-},\]
as claimed.
\end{proof}

\begin{proof}[Proof of Theorem \ref{t:NT}]
Let us define 
\[
\eta=\frac{2}{3}
\quad \text{ and } \quad
\omega=\frac{(3+\delta)}{c_1}
\sum_{i=1}^N d_i \ve_i 
,\]
where $\delta>0$ is to be determined. 
Now let $\ve$ be any positive constant.
Taking $\delta$ sufficiently small
compared to $\epsilon$ 
we 
see that the exponent of $K_\c{R}$ in the second term of~\eqref{eq:ee11} is
\begin{align*} 
1+4\omega(1+\eta+2\delta+\hat \ve)
&\leq 
1+\ve+\frac{4}{c_1}\Big(\sum_{i=1}^N d_i \ve_i\Big) (5+3 \hat \ve )
,\end{align*}
where $\hat \ve=\max\{\ve_1,\dots,\ve_N\}$.
Thus $E^{(I)}(\c{R})$ makes a 
satisfactory contribution 
for Theorem~\ref{t:NT}.
Taking $\delta$ sufficiently small allows us to check 
that 
\[
\frac{\eta \omega}{2}>\frac{1}{c_1} \sum_{i=1}^N d_i \ve_i
\text{  and }
(1-\eta)\omega >3\delta \omega +\frac{1}{c_1} \sum_{i=1}^N d_i \ve_i,
\]
for our choice of $\eta$ and $\omega$. This therefore  
shows that the first term in the right hand side of~\eqref{eq:ee22} and~\eqref{eq:ee33}
is $$\ll \frac{V}{(\log V)^{N+1}q_G},$$
which, owing to
$h_W^*(q_G)\geq 1$ and   $E_{f_i}(V;W)\geq 1$,
is satisfactory for Theorem~\ref{t:NT}.
A straightforward calculation now shows that 
for sufficiently small $\delta$
the contribution of the second terms on the right of~\eqref{eq:ee22},~\eqref{eq:ee33} and~\eqref{eq:ee44}
is also satisfactory for Theorem~\ref{t:NT}.
\end{proof}


\begin{thebibliography}{99}

\bibitem{nair}
 R. de la Bret\`eche and T. D. Browning,
Sums of arithmetic functions over values of binary forms.
{\em Acta Arith.} {\bf 125} (2006), 291--304.

\bibitem{moyennes}
 R. de la Bret\`eche and G. Tenenbaum,
Moyennes de fonctions arithm\'etiques de formes binaires. {\em Mathematika} 
{\bf 58} (2012), 290--304.

\bibitem{annals}
 R. de la Bret\`eche, T. D. Browning and E. Peyre,
On Manin's conjecture for a family of Ch\^atelet surfaces.
{\em Annals of Math.} {\bf 175} (2012), 297--343.

\bibitem{dreamteam}  
T.D. Browning and E. Sofos, 
Counting rational points on quartic del Pezzo surfaces with a rational conic.
{\em Submitted}, 2016.  (\texttt{arXiv:1609.09057})

\bibitem{KH}
K. Henriot, Nair-{T}enenbaum bounds uniform with respect to the
              discriminant. 
{\em Math. Proc. Cambridge Philos. Soc.} {\bf 152} (2012), 405--424.

\bibitem{iwko}
H. Iwaniec and E.  Kowalski, {\em Analytic number theory}.
American Math.\ Soc.\, Providence, RI, 2004.

\bibitem{Nair} M. Nair,
Multiplicative functions of polynomial values in short intervals.
{\em Acta Arith.} {\bf 62} (1992), 257--269.

\bibitem{NT}
M. Nair and G. Tenenbaum. Short sums of certain arithmetic functions. 
{\em Acta Math.} {\bf 180} (1998), 119--144.

\bibitem{shiu}
P. Shiu, A {B}run--{T}itchmarsh theorem for multiplicative functions.
{\em J.\ reine angew.\ Math.}  {\bf 313} (1980), 161--170.

\bibitem{sofos}
E. Sofos, Serre's problem on the density of isotropic fibres in conic bundles. {\em Proc.\ London Math.\ Soc.} {\bf 113} (2016), 1--28.

\end{thebibliography}
\end{document}